\documentclass[11pt,a4paper]{article}
\usepackage[dvips]{graphicx}
\usepackage{amsmath}
\usepackage{amsthm}
\usepackage{amssymb,amscd}
\usepackage{graphicx}
\usepackage{latexsym}
\usepackage{color}

\newtheorem{theorem}{Theorem}[section]
\newtheorem{def.}[theorem]{Definition}
\newtheorem{cor.}[theorem]{Corollary}
\newtheorem{conj.}[theorem]{Conjecture}
\newtheorem{lemma}[theorem]{Lemma}
\newtheorem{proposition}[theorem]{Proposition}

\newcommand{\norm}[2]{
\left\| #2 \right\|_{#1}
}

\def\ZZ{{\mathbb{Z}}}
\def\Z{{\mathbb{Z}}}
\def\RR{{\mathbb{R}}}

\def\RRd{{\mathbb{R}^d}}

\def\LtRd{{L^2(\mathbb{R}^d)}}

\def\Hil{{{\mathcal H}}}

\def\FF{\mathcal{F}}

\newcommand{\clsp}[1]{
{ \overline{span} \left\{ #1 \right\} }
}

\definecolor{darkviolet}{rgb}{0.58,0,0.83} 

\definecolor{blue}{rgb}{.134,.3,.6}

\newcommand\widecheck[1]{
\mathcal{F} #1 \mathcal{F}^{-1}
}

\newcommand\widehatxxl[1]{
\mathcal{F}^{-1} #1 \mathcal{F}
}

\title{An operator based approach to irregular frames of translates}
\author{P. Balazs\thanks{
Acoustics Research Institute, Austrian Academy of Sciences,
Wohllebengasse 12-14, 1040 Wien, Austria. E-mail:
peter.balazs@oeaw.ac.at.} \, and  S. Heineken\thanks{Departamento de
Matem\'atica and IMAS-CONICET, Facultad de Ciencias Exactas y Naturales, Universidad
de Buenos Aires, Pabell\'on I, Ciudad Universitaria, C1428EGA
C.A.B.A., Argentina, tel/fax:+541145763335. E-mail:
sheinek@dm.uba.ar (corresponding author).}}

\begin{document}
\date{}
\maketitle

\begin{abstract}
We consider translates of functions in $L^2(\RRd)$ along an
irregular set of points. Introducing a notion of pseudo-Gramian
function for the irregular case, we obtain conditions for a family
of irregular translates to be a Bessel sequence or Riesz sequence.

We find a representation of the canonical dual of a frame sequence
$\{\phi( \cdot - \lambda_k)\}_{k\in \ZZ}$ - where $\phi$ is a bandlimited function - in terms of its Fourier transform.
\end{abstract}



\section{Introduction}

Frames of translates, i.e. frames originated by the shifts of a given, fixed function,
are an important mathematical background for Gabor \cite{ole1},
wavelet \cite{Bentreib01} and sampling theory \cite{aldrgroech1}.
One generating function $\phi$ is shifted to create the analyzing family of elements, $\left\{ \phi \left( \cdot - \lambda_k \right) \right\}_{k\in \ZZ}$ for a sequence $\Lambda=\{\lambda_k\}_{k\in \ZZ}.$
These type of frames play a main role in the theory of shift invariant spaces (SIS) \cite{bo00,dedero94, dedero94b}, which is very useful for the modeling of problems in signal processing, and is central in approximation, sampling and wavelet theory.

Frames of translates are closely related to frames of exponentials also called {\sl Fourier frames} \cite{duffschaef1,Seip1,OS02}.

 The regular shifts, e.g. when $\lambda_k = k b$ for some $b>0$, were studied e.g. in \cite{Cachka01}. Investigations of irregular frames of translates, where the set $\Lambda$ has no regular structure, were done e.g. in \cite{ole1,aldrgroech1}. But there are several results that have not been generalized, in particular the concept of Gramian function, which is a fundamental tool for describing properties of systems of translates. Also no explicit formulas for the canonical dual of irregular frames of translates were given so far. We do this in the present paper.

The paper is structured as follows. In Section~\ref{sec:notprel0} we
summarize basic notations and preliminaries. In
Section~\ref{mainres0} we present the main results. In
Section~\ref{sec:gramirreg0} we generalize the concept of the
Gramian function and use it to describe Bessel and Riesz properties
of irregular shifts. In Section \ref{sec:operbaseirreg0} we look at
the interrelation of the frame-related operators for those sequences
to those of exponential functions.  This permits us to study
Bessel, frame and Riesz properties of irregular frames of
translates, and furthermore give formulas for the canonical dual,
and for equivalent Parseval frames.

\section{Notation and Preliminaries} \label{sec:notprel0}

Given $\lambda \in\RRd,$ we denote by $e_{\lambda}$ the function defined by $e_{\lambda}(x)=e^{-2\pi i\lambda x}.$
The operators $T_{\lambda}$ and $\mathcal{M}_{\lambda}$  are given by $T_{\lambda}f(x)=f(x-\lambda)$ and $\mathcal{M}_{\lambda}f(x)= e_{- \lambda} f(x)$ respectively.  We write $\hat{f}$ for the Fourier transform,
$\mathcal F : \LtRd \rightarrow \LtRd$, given by $\mathcal F (f) (w) = \hat{f}(w)=\int_{\RRd}f(x) e^{-2\pi i  x w}dx,$ for $f\in L^1(\RRd)\cap L^2(\RRd)$ with the natural extension to $L^2(\RRd).$



In this paper $E \subseteq \RRd$ is a bounded set.
We denote
$$PW_E = \left\{ f \in \LtRd \left|\right. \mbox{supp} (\hat f) \subseteq E  \right\}.$$

The space
\begin{equation*}\label{sec:embedltwo1}
\left\{ f \in L^2(\RRd) : f(x) = 0 \text { for } a.e. \ x \in \RRd \backslash E \right\},
\end{equation*}
is a closed subspace of $L^2(\RRd)$ which is isomorphic to $L^2(E)$, and we will identify these two spaces.
The Fourier transform maps $PW_E$ onto $L^2(E).$

When we write that $\{e_{\lambda_k}\}_{k\in \Z}$ is a frame (frame
sequence, Bessel sequence, or Riesz basis) of $L^2(E)$  we will mean
that the set $\{e_{\lambda_k} \chi_E\}_{k\in \Z}$ has that property,
where  $\chi_E$ stands for the indicator function of $E$, i.e. it is
an outer frame, Bessel or Riesz basis \cite{alcamo04}. For a given
operator $O$ we denote its pseudo-inverse by $O^\dagger$. Motivated
by that, for a function $\psi$ we denote
$$\psi^\dagger (x) = \left\{ \begin{array}{c c} 1 / \psi(x) & x \in supp (\psi) \\ 0 & \mathrm{otherwise} \end{array} \right.$$

%
%
%

Throughout this work $\Lambda=\{\lambda_k\}_{k\in \ZZ}$ will be a sequence in $\RRd.$
For a function $f$ we denote its range by $R_f.$

\subsection{Frames}

Let $\mathcal{H}$ be a separable Hilbert space.
A sequence $\{\psi_k\}_{k\in \ZZ}\subseteq \mathcal{H}$ is a {\sl frame} for $\mathcal{H}$ if there exist positive
constants $A$ and $B$ that satisfy
$$A \|f\|^2\leq\sum_{k\in \ZZ}|\langle f,\psi_k\rangle|^2 \leq
B\|f\|^2\,\,\,\,\,\forall f \in \mathcal{H}.$$ We denote the optimal
bounds by $A^{(opt)}$ and $B^{(opt)}$. If $A=B$ then it is called a
{\sl tight frame} and if $A=B=1$ a {\sl Parseval frame}. If
$\{\psi_k\}_{k\in \ZZ}$ satisfies the right inequality in the above
formula, it is called a {\sl Bessel sequence}.

 A sequence $\{\psi_k\}_{k\in \ZZ}\subseteq \mathcal{H}$ is a {\it Riesz basis} for
$\mathcal{H}$ if it is complete in $\mathcal{H}$ and there exist $0<A\leq B$
such that for every finite scalar sequence
$\{c_k\}_{k\in\ZZ}$
\begin{equation*}
A \sum |c_k|^2 \leq \left\|\sum_{k\in\Z} c_k \psi_k\right\|^2\leq
B\sum |c_k|^2\,.
\end{equation*}
The constants $A$ and $B$ are called {\it  Riesz bounds}.

We say that $\{\psi_k\}_{k\in \Z}\subseteq \mathcal{H}$ is a {\it frame sequence}
({\it Riesz sequence}) if it is a frame (Riesz basis) for its span.

For $\Psi  = \{\psi_k\}_{k\in \ZZ}$ a Bessel sequence, its {\sl analysis operator} $C: \mathcal{H}\rightarrow l^2(\ZZ)$ is defined by $C(f) =  \left\{ \left< f , \psi_k \right>_{\Hil} \right\}_{k\in \ZZ} $, and its {\sl synthesis operator} $D:l^2(\ZZ) \rightarrow \mathcal{H}$ by
$ Dc = \sum \limits_{k\in \ZZ} c_k \psi_k$.

If $\Psi$ is a frame of $\mathcal{H}$, the {\sl frame operator} $S: \mathcal{H}\longrightarrow \mathcal{H}$ given by
$S(f)= D_{\Psi}C_{\Psi} (f)=\sum \limits_{k\in \ZZ} \left< f , \psi_k \right> \psi_k$ is bounded, invertible, self-adjoint and positive.
There always exists a {\sl dual frame} of  $\Psi,$  which is a frame $\{\phi_k\}_{k\in\ZZ}$  such
that
\begin{equation}\label{dualeq}
f=\sum_{k\in\ZZ}\langle f,\phi_k\rangle \psi_k \,\,\,\forall f\in \mathcal{H}~~~~
\hbox{or}~~~~f=\sum_{k\in\ZZ}\langle f,\psi_k\rangle  \phi_k\,\,\,\forall f\in
\mathcal{H}.
\end{equation}
The sequence $\{S^{-1}\psi_k\}_{k\in \ZZ}$  satisfies
(\ref{dualeq}). It is called the {\sl canonical dual frame } of $
\{\psi_k\}_{k\in \ZZ}$ and we will denote it by
$\{{\widetilde{\psi_k}}\}_{k\in \ZZ}$.   We distinguish between the
operators by subscripts, e.g. $C_e$ is the analysis operator for the
set of exponentials $\{ e_{\lambda_k} \}_{k\in \ZZ}$ and $C_\phi$
those of the system of translates $\{ T_{\lambda_k} \phi \}_{k\in
\ZZ}$. For the operators connected to the canonical dual we will
write $\tilde C_e$ respectively $\tilde C_\phi$.

To every frame $\Psi=\{\psi_k\}_{k\in\ZZ}$ a {\sl canonical tight frame} can be associated, which is the sequence $\{S^{-\frac{1}{2}}\psi_k\}_{k\in \ZZ},$ where $S^{-\frac{1}{2}}$ is the positive square root of $S^{-1}.$

The (bi-infinite) {\sl Gram matrix} $G$ for $\Psi$ is given by   
$\left[ G \right]_{j,m} = \left< \psi_m , \psi_j \right>,$ for $j,m \in \ZZ.$
This matrix defines an operator by the standard matrix multiplication on $l^2(\ZZ)$, the set of square summable sequences.
It is known, see e.g. \cite{xxlstoeant11}, that this operator is bounded, if and only if the sequence forms a Bessel sequence. It is invertible on all of $l^2(\ZZ)$ if and only if the sequence is a Riesz sequence.


\subsection{Multipliers} \label{sec:Ltmultapp1}

In this work the operator that consists of a multiplication by a given function is important.

%
%
%

For $\LtRd$ we can give the following result, which is either known, see e.g. \cite{conw1} or can be easily proved:
\begin{proposition} \label{sec:LTmult1} Let $M_\phi : \LtRd \rightarrow \LtRd$ be the operator $\left( M_\phi (f) \right) (t) = \phi(t) f (t) $.
\begin{enumerate}
\item  $\phi \in L^\infty$  if and only if $M_\phi$ is bounded. In this case $\| \phi \|_\infty = \| M_\phi \|_{Op}$.
%
\item Assume there exist $p,P>0$ such that $ p \le | \phi(x) | \le P $ for almost all $x \in supp{ (\phi)}$. Then $Im ( M_\phi) = L^2( supp(\phi)),\,\, ker (M_\phi) = L^2(\RRd \backslash supp(\phi))$ and $M_\phi^\dagger = M_{\phi^\dagger}$.
On $supp (\phi)$ we have
$$ \norm{}{M_\phi} \ge p \norm{}{f}.$$
\item $M_\phi$ is boundedly invertible if and only if there exists $p>0$  such that $p \le | \phi(x) |$ a.e.
\end{enumerate}
\end{proposition}

\section{Main results} \label{mainres0}

For regular frames of translates the Gramian function is a very
important concept. For the case of irregular frames of translates,
where we have no group structure, we introduce the following notion.
\begin{def.} \label{sec:defgramirreg1} Let $\phi \in L^2(\RR^d).$ We define its {\em pseudo-Gramian function} by
$$ \Phi  =  \left| \hat \phi \right|^2  $$
\end{def.}
Surprisingly, as we will see in Section \ref{sec:gramirreg0} and Section~\ref{sec:operbaseirreg0}, this definition leads to quite analogous results as in the regular case, where this function is used only in a periodized version.

Among them we show the following:
\begin{theorem}\label{pe_new} Let $\phi \in \LtRd$ such that  $supp(\hat \phi)$ is compact. If one of the following conditions is fulfilled, the other two are equivalent
\begin{enumerate}
\item
\begin{enumerate}
\item $\{e_{\lambda_k}\}_{k\in \ZZ}$ is a frame for $L^2(supp(\hat \phi))$.
\item $\{ T_{\lambda_k} \phi\}_{k\in \ZZ}$ is a frame for $PW_{supp(\hat \phi)}$.
\item There exist $p,P > 0$ such that $p \le \Phi \le P$ a.e. on supp$(\hat \phi)$.
\end{enumerate}
\item
\begin{enumerate}
\item $\{e_{\lambda_k}\}_{k\in \ZZ}$ is a frame sequence in $L^2(supp(\hat \phi))$.
\item $\{ T_{\lambda_k} \phi\}_{k\in \ZZ}$ is a frame sequence in $PW_{supp(\hat \phi)}$.
\item There exist $p,P > 0$ such that $p \le \Phi \le P$ a.e. on supp$(\hat \phi)$.
\end{enumerate}
\item
\begin{enumerate}
\item $\{e_{\lambda_k}\}_{k\in \ZZ}$ is a Riesz sequence in $L^2(supp(\hat \phi))$.
\item $\{ T_{\lambda_k} \phi\}_{k\in \ZZ}$ is a Riesz sequence in $PW_{supp(\hat \phi)}$.
\item There exist $p,P > 0$ such that $p \le \Phi \le P$ a.e.
\end{enumerate}
\end{enumerate}
\end{theorem}

Furthermore we can give an explicit form for the canonical dual of an irregular set of translates, as a nice generalization of results in the regular case:
\begin{theorem} \label{sec:irregdual1} Let $\phi \in \LtRd$ such that  $supp(\hat \phi)$ is compact. Assume that there exist $p,P>0$ such that $p \le \Phi \le P $ a.e. on $supp(\hat \phi)$ and let $\{e_{\lambda_k}\}_{k\in \ZZ}$  be a frame sequence in $L^2(supp(\hat \phi)).$
Then the canonical dual of $\{ T_{\lambda_k} \phi\}_{k\in \ZZ}$ is $\{\theta_k\}_{k\in \ZZ}$
where
$$ \hat \theta_k = \left\{
\begin{array}{c c}
\frac{\hat \phi}{\Phi} \widetilde{e_{\lambda_k}} & \mbox{ on } supp(\hat \phi) \\
0 & \mbox{otherwise}.
\end{array}
\right. $$
In particular if $\{e_{\lambda_k}\}_{k\in \ZZ}$  is an
$A$-tight frame, then $\theta_k = T_{\lambda_k} \theta$ with $\hat
\theta = \frac{1}{A} \frac{1}{\overline{\hat \phi}}$ (on $supp(\hat
\phi)$).
\end{theorem}

\section{Gramian function for irregular frames of translates} \label{sec:gramirreg0}

In this section we will show some results about sufficient and necessary conditions on the pseudo-Gramian function for Bessel sequence and Riesz basis properties.

As a connection of the  pseudo-Gramian function with the Gramian matrix we get the following result.
\begin{proposition} \label{sec:irreggram1} Let $\phi \in L^2(\RR^d).$ For any system of translates $\{T_{\lambda_k} \phi \}_{k\in \ZZ}$ the Gramian matrix $G$ has the entries
$$ G_{k,l} = \hat \Phi \left(\lambda_k - \lambda_l\right). $$
\end{proposition}
\begin{proof}
$$ G_{k,l} = \left< T_{\lambda_k} \phi, T_{\lambda_l} \phi \right> = \left< \phi, T_{\lambda_l - \lambda_k} \phi \right> = $$
$$ = \left< \hat \phi, M_{\lambda_k - \lambda_l} \hat \phi \right> = \int \limits_{\RRd} \left| \hat \phi ( \xi ) \right|^2 e^{- 2 \pi i \left( \lambda_k - \lambda_l \right) \xi} d \xi = \left( \FF {\left| \hat \phi \right|^2} \right) \left( \lambda_k - \lambda_l \right) $$
\end{proof}

As a consequence of Schur's test, see e.g. \cite{gr01} and the properties of the Gram matrix we get:
\begin{cor.} \label{sec:schurtransl1}
Let $\phi \in L^2(\RR^d).$ If
$$\sup \limits_{k\in \Z} \sum_{l\in\Z} \left|  \hat \Phi \left(\lambda_k - \lambda_l\right) \right| < \infty ,$$
 then the system of translates $\{T_{\lambda_k} \phi \}_{k\in\ZZ}$ is a Bessel sequence.
\end{cor.}

On the other hand using \cite{cron71} the following can be shown:

\begin{cor.} If  $\phi \in L^2(\RR^d)$ and $\{T_{\lambda_k} \phi \}_{k\in\ZZ}$ is a Bessel sequence, then
$$\sup \limits_{k\in \Z}\sum_{l\in\Z} \left|  \hat \Phi \left(\lambda_k - \lambda_l\right) \right|^2 < \infty .$$
\end{cor.}

For Riesz bases we obtain the following result:
\begin{proposition}
Let $\phi\in L^2(\RR^d), \,\phi \not=0$ .
 If
$$ \norm{L^2(\RR^d)}{\phi}^2 > \sup \limits_{j\in \Z} \left(  \sum \limits_{l,l\not=j} \left| \hat \Phi \left( \lambda_j - \lambda_l \right) \right| \right),$$
then the system of translates $\{T_{\lambda_k} \phi \}_{k\in\ZZ}$ is a Riesz sequence.
\end{proposition}
\begin{proof}
As $\phi \in L^2(\RR^d)$ we have
$$ \hat \Phi (0) = \int_{\RR^d}\Phi (\omega) d\omega = \int_{\RR^d} \left| \hat \phi (\omega) \right|^2 d \omega = \norm{}{\hat \phi}^2 = \norm{}{\phi}^2.$$
Furthermore, by assumption
$$ \sup \limits_{k\in \Z} \sum \limits_{l\in\Z} \left|  \hat \Phi \left(\lambda_k - \lambda_l\right) \right| =
\left| \hat \Phi (0) \right| + \sup \limits_{k\in \Z} \sum \limits_{l, l\not=k} \left|  \hat \Phi \left(\lambda_k - \lambda_l\right) \right|
 < 2 \cdot \left| \hat \Phi (0) \right|.$$
The diagonal dominance (see e.g. \cite{gr01}) implies the invertibility of the Gram matrix operator on all of $l^2(\ZZ)$, which gives
the desired result.
\end{proof}

\section{The operator-based approach to irregular frames of translates} \label{sec:operbaseirreg0}

%

\subsection{Irregular translates in $PW_E$} \label{sec:operirregtrans1}

From now on we will assume $\phi \in PW_E.$

We define the subspaces in $L^2(E)$:
\begin{itemize}
\item $H_\Lambda = \clsp{ e_{\lambda_k} \cdot \chi_E: k\in\Z}$.
\item $F_\Lambda = \left\{ f \in H_\Lambda \left| \hat{\phi}  \cdot f \equiv 0\right. \right\} .$
\item $E_\Lambda = H_\Lambda \backslash F_\Lambda$.
\item $V = \clsp{ T_{\lambda_k} \phi:k\in\Z}$.
\end{itemize}

$H_\Lambda$, $F_\Lambda$ and $V$ are clearly closed subspaces. If $\{  e_{\lambda_k} \cdot \chi_E \}$ is a complete set, $H_\Lambda = L^2(E)$.

Analogous as in \cite[Theorem~4.1, Prop. 3.6]{xxlhein10} the following can be proved.
\begin{lemma} \label{sec:besbesirreg1}
\begin{enumerate}
\item Assume $\{e_{\lambda_k} \}_{k\in \ZZ}$ is a Bessel sequence of $L^2(E)$  with bound $B_e$ and 
there exists $P>0$ such that $\Phi ( \omega ) \le P$ a.e.
Then $\{ T_{\lambda_k} \phi\}_{k\in \ZZ}$ is a Bessel sequence for $L^2(\RR^d)$ with bound $B_\phi = B_e \cdot P$.
\item Let $\{ T_{\lambda_k} \phi\}_{k\in \ZZ}$ be a Bessel sequence for $L^2(\RR^d)$ with bound $B_\phi,$ and assume there exists $p>0$ such that $\Phi ( \omega ) \ge p$ a.e. in $E$. Then $\{e_{\lambda_k} \}_{k\in \ZZ}$ is a Bessel sequence for $L^2(E)$ with bound $B_e = B_\phi / p$.
\end{enumerate}
\end{lemma}

The following result generalizes Lemmas 7.2.1 and 7.3.2 in  \cite{ole1}.
\begin{lemma} \label{sec:exptranslat1}
Let $\{e_{\lambda_k}\}_{k\in \ZZ}$ be a Bessel sequence for $L^2(E)$ and assume there exists $P>0$ such that $\Phi (\omega) \le P$ a.e.  Let $\{c_k\}_{k\in\ZZ} \in \ell^2(\ZZ).$ Then $\sum_{k\in\ZZ} c_k T_{\lambda_k} \phi$ converges unconditionally in $L^2(\RR^d)$ and $\sum_{k\in\ZZ} c_k e_{\lambda_k}$ converges unconditionally in $L^2(E)$ and
$$ \mathcal F \left( \sum_{k\in\Z} c_k T_{\lambda_k} \phi \right) = \left( \sum_{k\in\Z} c_k e_{\lambda_k} \right) \hat \phi .$$
Therefore $H_\Lambda \cdot \hat \phi = \FF (R_{D_\phi})$. 
\\
$f \in R_{D_\phi}$ (and $f \not= 0$) if and only if there exists $F \in H_\Lambda$ ($F \in E_\Lambda$, respectively) such that $f =  F \cdot \hat \phi$.

\end{lemma}
\begin{proof} Following Lemma \ref{sec:besbesirreg1} $\{ T_{\lambda_k} \phi\}_{k\in\ZZ}$ is a Bessel sequence and so all involved sums converge unconditionally.
$$ \mathcal F \left( \sum_{k\in \ZZ} c_k T_{\lambda_k} \phi \right) =  \sum_{k\in \ZZ} c_k e_{ \lambda_k} \hat \phi .$$
As $\hat \phi$ is bounded the multiplication operator $M_{\hat \phi}$  is also bounded and therefore
$$ \left( \sum_{k\in \ZZ} c_k e_{\lambda_k} \right) \hat \phi = M_{\hat \phi} \left( \sum_{k\in \ZZ} c_k e_{ \lambda_k} \right) =  \sum_{k\in \ZZ} c_k M_{\hat \phi} \left( e_{ \lambda_k} \right) =  \sum_{k\in \ZZ} c_k e_{ \lambda_k} \hat \phi .$$
\end{proof}

\begin{cor.}
\label{sec:exptranslat2} Let $\{e_{\lambda_k}\}_{k\in \ZZ}$ be a Bessel sequence in $L^2(E)$ and assume there exists $P>0$ such that $\Phi \le P$ a.e.
Let $\{T_{\lambda_k} \phi\}_{k\in \ZZ}$ be a frame sequence.
Then $H_\Lambda \cdot \hat \phi = \clsp{ e_{\lambda_k} \hat
\phi:k\in \ZZ} = \hat V$ and so
 $f \in V$ (and $f \not= 0$) if and only if there exists $F \in
H_\Lambda$ ($F \in E_\Lambda$, respectively) such that $f =  F \cdot
\hat \phi$.

Furthermore 
for all $f \in V$ we have
$$ f \not = 0 \text{ if and only if } \sum_{k\in\Z} \left< f , \widetilde{T_{\lambda_k} \phi} \right> e_{\lambda_k} \in E_\Lambda. $$
\end{cor.}
\begin{proof} In this case $V = R_D$. We see from the proof of Lemma \ref{sec:exptranslat1} that $F = \sum \limits _{k\in\Z} \left< f , \widetilde{T_{\lambda_k} \phi} \right> e_{\lambda_k}$.
\end{proof}

\subsection{Relation of operators}

In this section we state the relations of the operators of the set of exponentials and the system of translates.
\begin{lemma} \label{sec:relopirr1}
Let $\{e_{\lambda_k}\}_{k\in \ZZ}$ be a Bessel sequence of $L^2 (E)$ and assume that there exists  $P> 0$ such that $\Phi (\omega) \le P$ a.e.
Then $\left\{ T_{\lambda_k} \phi \right\}$ is a Bessel sequence of $L^2(E)$ and
$$ \begin{array}{l l l}  D_\phi = \mathcal F^{-1} M_{\hat \phi} D_e, & \hspace{10mm} & C_\phi  = C_e M_{\overline{\hat \phi}} \mathcal F, \\
S_\phi = \widehatxxl{M_{\hat \phi} S_e M_{\overline{\hat \phi}},} & \hspace{10mm} & G_\phi = C_e M_{\Phi} D_e.
\end{array} $$
\end{lemma}
\begin{proof}
By Lemma \ref{sec:besbesirreg1} $\{ T_{\lambda_k} \phi\}_{k\in \ZZ}$ is a Bessel sequence for $PW_E$.
 For $c\in l^2(\Z),$ by Lemma~\ref{sec:exptranslat1},
$$  \mathcal F \left( D_\phi c \right) = \mathcal F \left({ \sum \limits_k c_k T_{\lambda_k} \phi }\right)= \left( \sum \limits_k c_k e_{\lambda_k} \right) \hat \phi  = M_{\hat \phi} D_e c. $$
For $f\in PW_E,$
$$ C_\phi f = \left\{\left< f , T_{\lambda_k} \phi \right>\right\}_{k\in\Z} = \left\{\left< \hat f , e_{ \lambda_k} \hat \phi \right> \right\}_{k\in\Z}= \left\{
\left< \hat f \overline{\hat \phi} , e_{ \lambda_k}  \right> \right\}_{k\in\Z}.$$
So
$$ S_\phi = D_\phi C_\phi = \mathcal F^{-1} M_{\hat \phi} D_e C_e M_{\overline{\hat \phi}} \mathcal F = \mathcal F^{-1} M_{\hat \phi} S_e M_{\overline{\hat \phi}} \mathcal F, $$
and
$$ G_\phi = C_\phi D_\phi = C_e M_{\overline{\hat \phi}} \mathcal F \mathcal{F}^{-1} M_{\hat \phi} D_e = C_e M_{\left| {\hat \phi} \right|^2} D_e .$$
\end{proof}

In particular we have that $R_{D_\phi} \subseteq \mathcal F^{-1} L^2(supp ( \hat \phi ))$.

\begin{lemma} \label{sec:relopirr2}
Let $\{e_{\lambda_k}\}_{k\in\Z}$ be a frame of $L^2(supp ( \hat \phi ))$, and assume there exist $p,P>0$ such that $p \le \Phi \le P$ a.e. on $supp ( \hat \phi )$. Then
$\left\{ T_{\lambda_k} \phi \right\}_{k\in \ZZ}$ is a frame sequence with span $PW_{supp ( \hat \phi )}$. Furthermore
$$ \begin{array}{l l l}  {S^{-1}_\phi}_{|_{PW_{supp ( \hat \phi )}} } = \widehatxxl{ \left( M_{\frac{1}{\overline{\hat \phi}}}  S_e^{-1} M_{\frac{1}{\hat \phi}}  \right)}, & \hspace{5mm} S_\phi^{\dagger} = \widehatxxl{ M_{\frac{1}{\overline{\hat{\phi}}}} S_e^{-1} M_{{\hat \phi}^\dagger}},
\end{array}$$
$$ \begin{array}{l l l}
D_e  =  M_{\frac{1}{\hat \phi}} \FF D_\phi ,  & \hspace{10mm} & C_e = C_\phi \FF^{-1} M_{\overline{\hat \phi}^\dagger},\\
{S_e}
 =  M_{\frac{1}{\hat \phi}} \FF S_\phi \FF^{-1} M_{\overline{\hat \phi}^\dagger}   , & \hspace{10mm} & G_e = C_\phi \widecheck{M_{\frac{1}{\Phi}}} D_\phi.
\end{array}$$
Moreover  $p A_e^{(opt)} \leq A_\phi^{(opt)} \leq
A_e^{(opt)} P$.
\end{lemma}
\begin{proof}
By Proposition~\ref{sec:LTmult1} $M_{\hat \phi}$ is invertible on $L^2(supp ( \hat \phi ))$, and therefore $ S_\phi = \FF^{-1} M_{\hat \phi} S_e M_{\overline{\hat \phi}} \FF$ is invertible on $PW_{supp ( \hat \phi )}$, and
$$ {S_\phi^{-1}}_{|_{PW_{supp ( \hat \phi )}} } = \FF^{-1} M_{\overline{\hat \phi}}^{-1} S_e^{-1} M_{\hat \phi}^{-1}  \FF = \widehatxxl{ M_{\frac{1}{\overline{\hat {\phi}}}} S_e^{-1} M_{\frac{1}{\hat \phi}} }.$$
 Note that $\FF$ maps $PW_E$ onto $L^2(E)$ and $M_{\hat \phi}$ maps $L^2(E)$ onto $L^2(supp ( \hat \phi ))$, with $M_{\hat \phi}^\dagger = M_{\hat \phi^\dagger}.$ Therefore $S_\phi^{\dagger} = \widehatxxl{ M_{\frac{1}{\overline{\hat{\phi}}}} S_e^{-1} M_{{\hat \phi}^\dagger} }$.

By Lemma~\ref{sec:exptranslat1}, $R_{D_\phi} \subseteq PW_{supp ( \hat \phi )}$ and since $\{e_{\lambda_k}\}_{k\in\ZZ}$ is a frame for $L^2(supp ( \hat \phi ))$, we know that $R_{D_\phi} = PW_{supp ( \hat \phi )}$. Now $D_\phi = \FF^{-1} M_{\hat \phi} D_e$ implies $M_{\hat \phi}^{-1} \FF D_\phi = D_e.$ And $R_{D_e} = L^2(supp ( \hat \phi ))$.

Therefore, $C_\phi \FF^{-1}    = C_e M_{\overline{\hat \phi}}$ and so $C_e = C_\phi \FF^{-1} \left( M_{\overline{\hat \phi}}\right)^\dagger = C_\phi \FF^{-1} M_{\overline{\hat \phi}^\dagger}$.

So
$$G_e = C_e D_e = C_\phi \FF^{-1} M_{\overline{\hat \phi}^\dagger} M_{\frac{1}{\hat \phi}} \FF D_\phi = C_\phi \FF^{-1} M_{\frac{1}{\left| \hat \phi \right|^2}} \FF D_\phi. $$

Finally,
$$\frac{1}{A_\phi^{(opt)}} = \norm{Op}{S_\phi^{-1}} = \norm{Op}{\mathcal F^{-1} M^{-1}_{\overline{\hat \phi}} S_e^{-1} M^{-1}_{{\hat \phi}} \mathcal F} \le  \frac{1}{p} \frac{1}{A_e^{(opt)}}.$$
And
$$\frac{1}{A_e^{(opt)}} = \norm{Op}{S_e^{-1}} = \norm{Op}{M_{\overline{\hat {\phi}}}   \FF S_\phi^{-1} \FF^{-1}   M_{\hat \phi}} \le P \frac{1}{A_\phi^{(opt)}}.$$

\end{proof}


Furthermore, we can state the following result about the exactness of sequences of translates and sequences of exponentials.
\begin{cor.} \label{sec:transexpexact1}
Let $\{e_{\lambda_k}\}_{k\in \ZZ}$ be a Bessel sequence of $L^2(E).$
\begin{enumerate}
\item Assume there exists $P>0$ such that $\Phi \le P$ a.e.. If $\{T_{\lambda_k}\phi\}_{k\in\ZZ}$ is exact, then  $\{e_{\lambda_k}\}_{k\in \ZZ}$ is exact.
\item \label{Extrf} Assume there exist $p,P>0$ such that $p \le \Phi \le P$ a.e. in $E.$ Then $\{T_{\lambda_k}\phi\}_{k\in \ZZ}$ is exact if and only if  $\{e_{\lambda_k}\}_{k\in \ZZ}$ is exact.
\end{enumerate}
\end{cor.}
\begin{proof}
Lemma~\ref{sec:relopirr1} yields $ker D_e \subseteq ker D_\phi$.

If $|\hat{\phi}|$ is bounded from above and below, then $M_{\hat{\phi}}$ is
invertible and so by Lemma~\ref{sec:relopirr2}, part \ref{Extrf} follows.
\end{proof}

%

We can now proof one of the theorems stated in Section \ref{mainres0}, a generalization of results in \cite{xxlhein10}.
\begin{proof}[Proof of Theorem \ref{pe_new}]

(1) Since $\hat \phi$ is bounded from above and from below, $M_{\hat \phi}$ is invertible and so by Lemma \ref{sec:relopirr1}
 $(c) \Longrightarrow \left( (b) \Longleftrightarrow (a) \right)$.

Theorem 4.1 in  \cite{xxlhein10} implies $(a) \Longrightarrow \left( (b) \Longleftrightarrow (c) \right)$ .

If $D_e$ and $D_\phi$ are surjective as well as $C_e$ and $C_\phi$ are injective,
$M_{\hat \phi}$ is bijective and so the last direction is shown.

(2) As the ranges of the operators correspond, the proof is the same as (1).

(3) follows from Corollary \ref{sec:transexpexact1} applied to $E=supp(\hat \phi)$ and $(1)$ .

\end{proof}
Observe that Proposition \ref{sec:irreggram1}  follows also from the equation
$ G_\phi =  C_\phi D_\phi = C_e M_{\Phi} D_e$ in Lemma~\ref{sec:relopirr1}.

Using Theorem~\ref{pe_new}, we can generalize  Proposition 7.3.6. in \cite{ole1} in the following sense:
\begin{cor.}  
Let $\{e_{\lambda_k}\}_{k\in \ZZ}$ be a Riesz sequence in $L^2(supp(\hat \phi))$. If $T_{\lambda_k} \phi$ forms an overcomplete frame sequence, then $\hat \phi$ is discontinuous and so $\phi \not\in L^1$.
\end{cor.}

\subsection{The canonical dual}

We arrive to the relation between the canonical duals of frame sequences of exponentials and frame sequences of translates stated in Theorem \ref{sec:irregdual1}.
\begin{proof}[Proof of Theorem \ref{sec:irregdual1}]
By Lemma~\ref{sec:relopirr2}
$$S_\phi^{-1} = \mathcal F^{-1} M_{\overline{\hat \phi}}^{-1} S_e^{-1} M_{{\hat \phi}}^{-1} \mathcal F $$ 
and therefore
$$S_\phi^{-1} T_{\lambda_k} \phi  = \mathcal F^{-1} M_{\frac{1}{\overline{\hat \phi}}} S_e^{-1} M_{\frac{1}{{\hat \phi}}} \mathcal F T_{\lambda_k} \phi  =\mathcal F^{-1} M_{\frac{1}{\overline{\hat \phi}}} S_e^{-1} M_{\frac{1}{{\hat \phi}}} e_{ \lambda_k} \hat \phi =$$
$$=\mathcal F^{-1} M_{\frac{1}{\overline{\hat \phi}}} S_e^{-1} e_{ \lambda_k} = \mathcal F^{-1} M_{\frac{1}{\overline{\hat \phi}}} \widetilde{e_{ \lambda_k}}=\mathcal F^{-1} M_\frac{\hat \phi}{\left| \hat \phi \right|^2} \widetilde{e_{\lambda_k}}  . $$
\end{proof}

From Lemma~\ref{sec:relopirr1} we obtain:
\begin{cor.} \label{sec:irregframoptight}
Assume $\{e_{\lambda_k}\}_{k\in\ZZ}$ is an $A_e$-tight frame of $L^2(E)$, and that there exist $P> 0$ such that $\Phi (\omega) \le P$ a.e.
Then $\left\{ T_{\lambda_k} \phi \right\}$ is a Bessel sequence of $L^2(E)$ and Then the frame operator of $\{T_{\lambda_k} \phi\}_{k\in\ZZ}$ is
$$ S_\phi = \mathcal \FF^{-1} M_{A_e \cdot \Phi} \mathcal \FF .$$
\end{cor.}

The following Corollary shows the relation between the analysis and synthesis operator of the duals of the shifts and of the exponentials.
\begin{cor.} Assume that there exist $p,P>0$ such that $p \le \Phi \le P $ a.e. on $supp(\hat \phi)$ and let $\{e_{\lambda_k}\}_{k\in \ZZ}$  be a frame sequence in $L^2(supp(\hat \phi)).$ Then
$$ C_{\theta} = \widetilde{C_\phi} = \widetilde{C_{e}} M_{\frac{\hat \phi}{\Phi}} \FF. $$
$$ D_{\theta} = \widetilde{D_\phi} = \FF^{-1} M_{\frac{\hat \phi}{\Phi}} \widetilde{D_{e}}.$$
\end{cor.}

Now let us investigate the operator $S^{1/2}$.
\begin{lemma}
\begin{enumerate}
\item Let $\Phi = 1$ a.e.
Then
$S_\phi^{-1/2} = \mathcal F^{-1} M_{\hat \phi} S_e^{-1/2} M_{\overline{\hat \phi}} \mathcal F$.
\item Let $\{e_{\lambda_k}\}_{k\in\ZZ}$ form an $A_e$-tight frame of $L^2(E)$. Then
$S^{-1/2} = \mathcal F^{-1} M_{\sqrt{A}|\hat \phi |} \mathcal F$.
\end{enumerate}
\end{lemma}
\begin{proof} 1. With the assumption $M_{\hat \phi}$ is a unitary operator. Therefore
$$\sqrt{S_\phi^{-1}} = \sqrt{\mathcal F^{-1} M_{\overline{\hat \phi}}^{-1} S_e^{-1} M_{{\hat \phi}}^{-1} \mathcal F } = \mathcal F^{-1} M_{\hat \phi} S_e^{-1/2} M_{\overline{\hat \phi}} \mathcal F.$$

 2. By Corollary \ref{sec:irregframoptight}, $S_\phi = \mathcal F^{-1} M_{A_e |\hat \phi |^2} \mathcal F$, and so $\sqrt{S_\phi} = \mathcal F^{-1} M_{\sqrt{A} |\hat \phi |} \mathcal F$.
\end{proof}

\begin{theorem} Let $\{e_{\lambda_k}\}_{k\in\ZZ}$  be a frame for $L^2(supp ( \hat \phi ))$, and assume there exist $p,P>0$ such that $p \le \Phi \le P $ a.e. on $supp(\hat \phi)$.
Set $$ \hat{\theta_k^\# }= \left\{
\begin{array}{c c}
\frac{\hat \phi}{\sqrt{\Phi}} e_{\lambda_k}^\# & \mbox{ on } supp(\Phi) \\
0 & \mbox{otherwise}
\end{array}
\right. , $$ where we denote the canonical tight frame of
$\{e_{\lambda_k}\}_{k\in \ZZ}$ by $\{e_{\lambda_k}^\#\}_{k\in \ZZ}$.
Then $\{\theta_k^\#\}_{k\in \ZZ}$ forms a Parseval frame with the
same span as $\{ T_{\lambda_k} \phi \}_{k\in \ZZ}$. The frames $\{
\theta_k^\# \}_{k\in \ZZ}$ and $\{ T_{\lambda_k} \phi \}_{k\in \ZZ}$
are equivalent.
%

If  $\{e_{\lambda_k}\}_{k\in\ZZ}$  forms an $A_e$-tight frame
, then
$\{ \theta_k^\# \}_{k\in \ZZ}$ forms a tight frame of translates with shifts $\lambda_k$ and generator
$\frac{1}{A_e} \FF^{-1} \frac{\hat \phi}{\sqrt{\Phi}}$.

\end{theorem}
\begin{proof}
By \cite[Theorem 4.1]{xxlhein10} $\theta_k^\#$ is a frame sequence. Clearly $\clsp{\theta_k^\#: k\in \ZZ} \subseteq  PW_{supp ( \hat \phi )}$.
For $f \in  PW_{supp ( \hat \phi )}$,

$$ \sum \limits_k \left< \hat f, \frac{\hat \phi}{\sqrt{\Phi}} e_{\lambda_k}^\#  \right> \frac{\hat \phi}{\sqrt{\Phi}} e_{\lambda_k}^\# = \left( \sum \limits_k \left< \overline{\frac{\hat \phi}{\sqrt{\Phi}}} \hat f,  e_{\lambda_k}^\#  \right> e_{\lambda_k}^\# \right) \cdot \frac{\hat \phi}{\sqrt{\Phi}} =  $$
$$ = \overline{\frac{\hat \phi}{\sqrt{\Phi}}} \hat f \cdot \frac{\hat \phi}{\sqrt{\Phi}} = \hat f. $$
Therefore $S = Id$ and $\clsp{\theta_k^\#: k\in \ZZ} = PW_{supp ( \hat \phi )}$. The equivalence follows from the fact that by assumption $M_{\frac{1}{\sqrt{\Phi}}}$ is a bijection and so $R_{C_{\theta_k^\#}} = R_{C_{T_{\lambda_k \phi}}}$.

The result follows from Lemma~\ref{sec:relopirr2}.


\end{proof}

\section*{Acknowledgements}
The work on this paper was partially supported by the Austrian Science Fund (FWF) START-project FLAME ({\it Frames and Linear Operators for Acoustical Modeling and Parameter Estimation}; Y 551-N13), and the FP7 project PIEF-GA-2008-221090. We also thank MINCyT (Argentina) and BMWF (Austria) for the support of the bilateral project {\it Shift Invariant Spaces And Its Applications To Audio Signal Processing}; AU/10/25.

\end{document}